\documentclass[10pt,a4paper]{article}
\usepackage{amssymb}
\usepackage{amsmath, enumerate,amsthm,wasysym}
\usepackage{epsfig}
\newtheorem{theorem}{Theorem}

\newtheorem{corollary}[theorem]{Corollary}

\newtheorem{definition}[theorem]{Definition}
\newtheorem{example}[theorem]{Example}

\newtheorem{lemma}[theorem]{Lemma}

\newtheorem{proposition}[theorem]{Proposition}
\newtheorem{remark}[theorem]{Remark}

\begin{document}
\title{ Products of free random variables and $k$-divisible non-crossing partitions}
\author{Octavio Arizmendi\footnote{Supported by Deutsche Forschungsgemeinschaft (DFG), Project SP419/8-1. E-mail: arizmendi@math.uni-sb.de} and Carlos Vargas\footnote{Supported by Mexican National Council of Science and Technology (CONACYT) ref. 214839/310129. E-mail: carlos@math.uni-sb.de} \\ Universit\"{a}t des Saarlandes, FR $6.1-$Mathematik,\\ 66123 Saarbr\"{u}cken, Germany \\ }
\date{\today}
\maketitle

\begin{abstract}
We derive a formula for the moments and the free cumulants of the
multiplication of  $k$ free random variables in terms of $k$-equal and $k$-divisible non-crossing partitions. This leads to a new simple proof for the bounds of the right-edge of the support of the free multiplicative convolution $\mu^{\boxtimes k}$, given by Kargin in \cite{Kar}, which show that the support grows at most linearly with $k$. Moreover, this combinatorial approach generalizes the results of Kargin since we do not require the convolved measures to be identical. We also give further applications, such as a new proof of the limit theorem of Sakuma and Yoshida \cite{SaYo}.
\end{abstract}

\section{Introduction and statement of results}

Until recently, $k$-divisible non-crossing partitions have been overlooked in free probability and have barely appeared in the literature.
However, their structure is very rich and there are, for instance, quite natural bijections between $k$-divisible non-crossing partitions and $(k+1)$-equal partitions which preserve a lot of structure, (see e.g. \cite{Ar2}). Furthermore, as noticed in \cite{Ar1}, the moments of $\mu^{\boxtimes k}$ can be computed using $k$-divisible non-crossing partitions.

In this paper we exploit the fact that $k$-divisible and $k$-equal partitions are linked, by the Kreweras complement, to partitions which are involved in the calculation of moments and free cumulants of the product of $k$ free random variables. For details on free cumulants and their relevance in free probability, see \cite{NiSp}.

Given $a,b\in\mathcal{A}$ free random variables, with free cumulants  $\kappa_{n}(a)$ and $\kappa_{n}(b)$, respectively, one can calculate the free cumulants of $ab$ by
\begin{equation}\label{2prod}
\kappa_n(ab)=\sum_{\pi \in NC(n)}\kappa_{\pi }(a)\kappa_{Kr(\pi )}(b),
\end{equation}
where $Kr(\pi)$ is the Kreweras complement of the non-crossing partition $\pi$. In particular, we are able to compute the free cumulants of the free multiplicative convolution of two compactly supported probability measures $\mu, \nu$, such that $Supp(\mu)\subseteq [0,\infty)$ by
\begin{equation}
\kappa_n(\mu \boxtimes \nu)=\sum_{\pi\in NC(n)}\kappa_{\pi
}(\mu)\kappa_{Kr(\pi )}(\nu) \label{cum-prod}.
\end{equation}
In principle, this formula could be inductively used to provide the free cumulants and moments of the convolutions of $k$ (not necessarily equal) positive probability measures. This approach, however, prevents us from noticing a deeper combinatorial structure behind such products of free random variables.

Our fundamental observation is that, when $\pi$ and $Kr(\pi)$ are drawn together, the  partition $\pi\cup Kr(\pi)\in NC(2n)$ is exactly the Kreweras complement of a $2$-equal partition (i.e. a non-crossing pairing). Furthermore, one can show using the previous correspondence that Equation (\ref{2prod}) may be rewritten as
\begin{equation}
\kappa_n(ab)=\sum_{\pi \in NC_2(n)}\kappa_{Kr(\pi)}(a,b,\dots,a,b),
\end{equation}
where $NC_2(n)$ denotes the $2$-equal partitions of $[2n]$.

Since $2$-equal partitions explain the free convolution of two variables, it is natural to try to describe the product of $k$ free variables in terms of $k$-equal partitions.

The main result of this work is the following.

\begin{theorem} \label{main}
Let $a_1,\dots ,a_k$ $\in (\mathcal{A},\tau)$ be free random variables. Then the free cumulants and the moments of $a:=a_1\dots a_k$ are given by
\begin{eqnarray}
\kappa_n(a)&=&\sum_{\pi \in NC_{k}(n)}\kappa_{Kr(\pi)}(a_1,\dots ,a_k), \label{kcum2}\\
\tau(a^n)&=& \sum_{\pi \in NC^k(n)}\kappa_{Kr(\pi)}(a_1,\dots ,a_k), \label{kmom2} 
\end{eqnarray}
where $NC_k(n)$ and $NC^k(n)$ denote, respectively, the $k$-equal and $k$-divisible partitions of $[kn]$.

\end{theorem}

The main application of our formulas is a new proof of the fact, first proved by Kargin \cite{Kar}, that for positive measures centered at $1$, the support of the free multiplicative convolution $\mu^{\boxtimes k}$ grows at most linearly. Moreover, our approach enables us to generalize to the case $\mu_1\boxtimes \dots \boxtimes \mu_k$, as follows.

\begin{theorem}
There exists a universal constant $C>0$ such that for all $k$ and any $\mu_1,\dots ,\mu_k$ probability measures supported on $[0,L]$, satisfying $E(\mu_i)=1$ and $Var(\mu_i)\geq\sigma^2$, for $i=1,\dots ,k$, the supremum $L_k$ of the support of the measure $\mu_1\boxtimes \dots \boxtimes \mu_k$ satisfies
 \[\sigma^2k\leq L_k<CLk.\] 
\end{theorem}

In other words, for (not necessarily identically distributed) positive free random variables $(X_i)_{i\geq 1}$ such that $E(X_i)=1$, $Var(X_i)\geq \sigma^2$ and $||X_i||\leq L$, $i\geq 1$,  we have that
 \[\lim \sup_{n\to \infty} n^{-1}||X_1^{1/2} \cdots X_{n-1}^{1/2}X_nX_{n-1}^{1/2} \cdots X_1^{1/2}||<CL\]
and
 \[\lim \inf_{n\to \infty} n^{-1} ||X_1^{1/2} \cdots X_{n-1}^{1/2}X_nX_{n-1}^{1/2} \cdots X_1^{1/2}||\geq\sigma^2.\]

Let us point out that for the case $\mu_1=\dots=\mu_k$, the previous theorem can be proved as using the methods of \cite{KeSp}.
However, the norm estimates given there are meant to address more general situations (where certain linear combinations of products are allowed) and hence, the constants obtained using these methods for our specific problem are far from optimal.

Finally, we shall mention, that when $\mu_1=\dots=\mu_k$, Kargin proved in \cite{Kar2} that \[\lim_{k\to\infty}k^{-1}L_k=e\sigma^2.\]

The paper is organized as follows: Section 2 includes preliminaries on non-crossing partitions and the basic definitions from free probability that will be required. Some enumerative aspects of $k$-divisible partitions and their Kreweras complement are included. In Section 3 we prove our main formulae for the products of free random variables. We also express our results in terms of free multiplicative convolutions of compactly supported positive probability measures. Our formulas are used in Section 4 to derive bounds for the support of free multiplicative convolutions. Section 5 gives further applications and examples of our formulas. We calculate the free cumulants of products of free Poissons and shifted semicirculars. We show that the $n$-th cumulant of $\mu^{\boxtimes k}$ eventually becomes positive as $k$ grows. We also provide a new proof of the limit theorem of Sakuma and Yoshida \cite{SaYo}.

\section*{Acknowledgments}
The authors would like to thank Professor Vladislav Kargin for pointing out reference \cite{Kar2}. They are grateful to Professor Roland Speicher for valuable comments and discussions.

\section{Preliminaries}

\subsection{Non-crossing partitions}

A \textit{partition} of a finite, totally ordered set $S$ is a decomposition $\pi =\{V_{1},\dots,V_{r}\}$ of $S$
into pairwise disjoint, non-empty subsets $V_{i}$, $(1\leq i\leq r)$, called \textit{blocks}, such that $V_{1}\cup V_{2}\dots \cup V_{r}=S$. The number of blocks of 
$\pi $ is denoted by $\left\vert \pi \right\vert$.

A partition $\pi =\{V_{1},\dots,V_{r}\}$ is called \textit{non-crossing} if for every $4$-tuple $a < b < c < d\in S$
such that $a,c\in V_{i}$ and $b,d\in V_{j}$, we have that $V_j=V_i$. We denote by $NC(S)$ the set of non-crossing partitions of $S$. For $S=[n]:=\{1, 2,\dots, n\}$ we simply write $NC(n)$.

$NC(S)$ can be equipped with the partial order $\preceq$ of reverse refinement ($\pi\preceq\sigma$ if and only if every block of $\pi$ is completely contained in a block of $\sigma$). This turns $NC(S)$ into a lattice, in particular, we can talk about the least upper bound $\pi_1\vee\pi_2\in NC(S)$ of two non-crossing partitions $\pi_1,\pi_2\in NC(S)$. 
We denote the minimum and maximum partitions of $NC(n)$ by $0_n$,  $1_n$, respectively.  We will write \[\rho_k^n:=\{(1,\dots,k),(k+1,\dots,2k),\dots,((k-1)n+1,\dots,kn)\}\] for the  partition in $NC(kn)$ consisting of $n$ consecutive blocks of size $k$.

\begin{definition}
We say that a non-crossing partition $\pi$ is $k$-divisible if the sizes of all the blocks of are multiples of $k$.  If, furthermore, all the blocks are exactly of size $k$, we say that $\pi$ is $k$-equal. A partition $\pi \in NC(nk)$ is called $k$-preserving if all its blocks contain numbers with the same congruence modulo $k$. A $k$-preserving partition $\pi\in NC(nk)$ is called $k$-completing if $\pi \vee \rho_k^n=1_{nk}$.
\end{definition}

Later we will see that these concepts are closely related. We denote the set of $k$-divisible non-crossing partitions of $[kn]$ by $NC^k(n)$ and the set of  $k$-equal  non-crossing partitions of $[kn]$ by $NC_k(n)$.

It is helpful to picture partitions via their circular representation: We think of $[n]$ as the clockwise labelling of the vertices of a regular $n$-gon.
If we identify each block of $\pi\in NC(n)$ with the convex hull of its
corresponding vertices, then we see that $\pi $ is non-crossing precisely
when its blocks are pairwise disjoint (that is, they don't cross). 

\begin{figure}
  \begin{center}
    \includegraphics[scale=0.4]{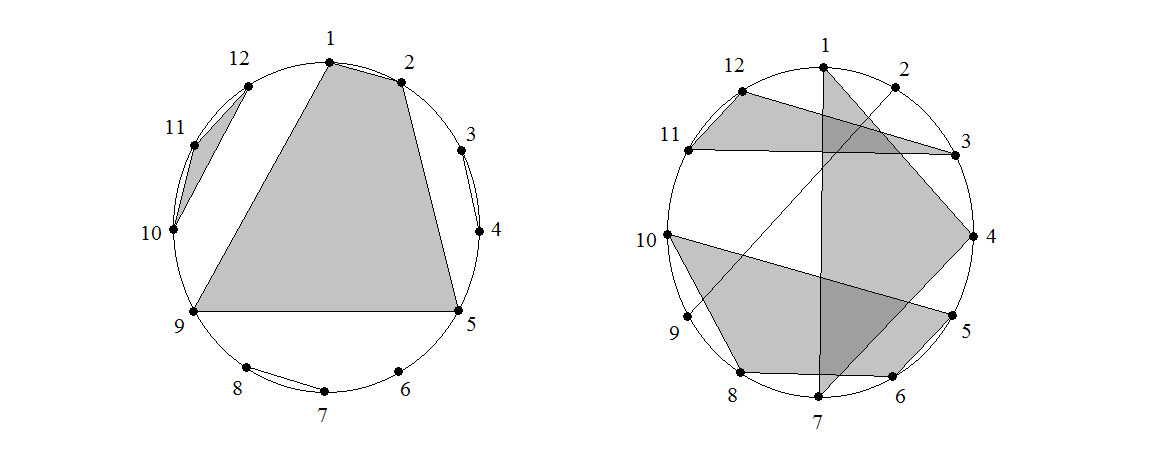}
    \caption{The non-crossing partition $
\{\{1,2,5,9\},\{3,4\},\{6\},\{7,8\},\{10,11,12\}\}$, and the crossing
partition $\{\{1,4,7\},\{2,9\},\{3,11,12\},\{5,6,8,10\}\}$ of the set $[12]$ in their circular representations.}
  \end{center}
\end{figure}
It is well known that the number of non-crossing partitions is given by the Catalan numbers $C_n:=\frac{1}{n+1}\binom{2n}{n}$.
The following formula (proved by Kreweras \cite{Kr}) counts the number of partitions of a given type.

\begin{proposition}
\label{NC type}Let $n$ be a positive integer and let $r_{1},r_{2},\dots, r_{n}\in 
\mathbb{N}
\cup \{0\}$ be such that $r_{1}+2r_{2}+\dots+nr_{n}=n$. Then the number of
partitions of $\pi $ in $NC(n)$ with $r_{1}$ blocks with $1$ element, $r_{2}$
blocks with $2$ elements, $\dots$, $r_{n}$ blocks with $n$ elements equals
\begin{equation}
\frac{n!}{\left(n+1-\sum\limits_{i=1}^{n}r_{i}\right)!\prod\limits_{i=1}^{n}r_{i}!}
.\label{rtype partitions}
\end{equation}
\end{proposition}

From the previous proposition one can easily count $k$-equal partitions.
\begin{corollary}
\label{PartexactK}
\begin{equation*}
|NC_k(n)|=\frac{\binom{kn}{n}}{(k-1)n+1}.
\end{equation*}
\end{corollary}
Finally, the number of $k$-divisible partitions was essentially given by Edelman \cite{Ede}.

\begin{proposition}
\label{k-divisible}
\begin{equation*}
|NC^k(n)|=\frac{\binom{(k+1)n}{n}}{kn+1}.
\end{equation*}
\end{proposition}

Out of two non-crossing partitions $\pi_1,\pi_2\in NC(n)$ we can build the $2$-preserving partition $\pi_1\cup\pi_2$ by thinking $\pi_1\in NC(\{1,3,\dots ,2n-1\})$, $\pi_2\in NC(\{2,4,\dots ,2n\})$, and drawing them together. In general, the resulting partition may have crossings. Then for a given $\pi\in NC(n)\cong NC(\{1,3,\dots ,2n-1\})$ we define its \textit{Kreweras complement} \[Kr(\pi):=\max\{\sigma \in NC(n)\cong NC(\{2,4,\dots ,2n\}) :\pi\cup\sigma\in NC(2n)\}.\]

The Kreweras complement \cite{Kr} satisfies many nice properties. The map $Kr:NC(n)\to NC(n)$ is an order reversing isomorphism. Furthermore, for all $\pi \in NC(n)$ we have that $|\pi|+|Kr(\pi)|=n+1$.

Similar to Proposition \ref{NC type}, one can count the number of partitions $\pi$, such that $\pi$ and $Kr(\pi)$ have certain block structures. Let $(r_i)_{1\leq i\leq n}$, $(b_j)_{1\leq j\leq n}$ be tuples satisfying 
\begin{eqnarray}
1r_1+2r_2+\dots+nr_n=n=1b_1+2b_2+\dots+nb_n,\\
|\pi|+|Kr(\pi)|=r_1+\dots+r_n+b_1+\dots+b_n=n+1. \label{Kblocks}
\end{eqnarray}
Then the number of partitions such that $\pi$ has $r_i$ blocks of size $i$ and
 $Kr(\pi)$ has $b_j$ blocks of size $j$ is given by the formula 
\begin{equation}\label{typekr}
n\frac{(|\pi|-1)!(|Kr(\pi)|-1)!}{r_1!\cdots r_n!b_1!\cdots b_n!}. 
\end{equation}
When $\pi$ is $k$-equal, Equation (\ref{typekr}),  reduces to \[k\frac{((k-1)n)!}{b_1!\cdots b_n!}.\]

As a consequence, we can show that for large $k$, the Kreweras complements of $k$-equal partitions have ``typically'' small blocks.
More precisely, for $n,k\geq 1$ let \[NC(k,n)_{2,1}:=\{\pi\in NC_k(n):Kr(\pi) \text{ contains only pairings and singletons}\}\subseteq NC_k(n).\]
In this case, the only possibility is that $b_1=n(k-2)+2$, $b_2=n-1$ and $b_i=0$ for $i>2$. Hence
\begin{equation} \label{sizeH2}
|NC(k,n)_{2,1}|=k\frac{((k-1)n)!}{((n(k-2)+2)!(n-1)!}.
\end{equation}
 An easy
application of Stirling's approximation formula shows that 
 \begin{equation} \label {H2}
\lim_{k\rightarrow \infty}\frac{|NC(k,n)_{2,1}|}{|NC_k(n)|}=1.
\end{equation}

\subsection{Free Probability and Free Cumulants}

A $C^*$\textit{-probability space} is a pair $(\mathcal{A},\tau)$, where $\mathcal{A}$ is a unital $C^*$-algebra and $\tau:\mathcal{A}\to\mathbb{C}$ is a positive unital linear functional.

The \textit{free cumulant functionals} $(\kappa_n)_{n\geq1}$, $\kappa_n:\mathcal{A}^n\to \mathbb{C}$ are defined inductively and indirectly by the moment-cumulant formula:
\begin{equation} \label{momcumfmla}
\tau(a_1\cdots a_n)=\sum_{\pi\in NC(n)}\kappa_{\pi}(a_1,\dots,a_n),
\end{equation}
where, for a family of functionals $(f_n)_{n\geq 1}$ and a partition $\pi\in NC(n)$, $\pi=\{V_1,\dots,V_i\}$, $V_j=\{m^{(j)}_1,\dots ,m^{(j)}_{|V_j|}\}\subseteq [n]$, we define
\begin{equation}
f_{\pi}(a_1,\dots,a_n):=\prod_{j=1}^if_{|V_j|}(a_{m^{(j)}_1},\dots,a_{m^{(j)}_{|V_j|}}).
\end{equation}

If we put $\tau_n(a_1,\dots ,a_n):=\tau(a_1\cdots a_n)$, Formula \ref{momcumfmla} may be inverted as follows:

\begin{equation}
\kappa_n(a_1,\dots ,a_n)=\sum_{\pi \in NC(n)}\tau_\pi(a_1,\dots,a_n)Mob[\pi,1_n],
\end{equation}

where $Mob:NC(n)\times NC(n)\to \mathbb{C}$ is the M\"obius function on $NC(n)$. For a detailed exposition we refer to \cite[Lecture 11]{NiSp}.

The main property of the free cumulants is that they characterize free independence. Here we will actually employ this characterization as the definition for freeness (for the equivalence with the usual definition see e.g. \cite{NiSp}).

The variables $a_1,\dots ,a_k\in \mathcal {A}$ are \textit{free} if and only if we have, for all $n\in \mathbb{N}$ and $1\leq i_1,\dots,i_n\leq k$, that \[\kappa_n(a_{i_1},\dots,a_{i_n})=0,\] unless $i_1=i_2=\dots=i_n$.

  The main tool used to derive our formulae is the following, proved by Krawczyk and Speicher \cite{KrSp}.
\begin{proposition}[Formula for products as arguments]\label{products}
Let $(\mathcal{A},\tau)$ be a $C^*$-probability space and let $(\kappa_n)_{n\geq1}$ be the corresponding free cumulants. Let $m,n\geq 1$, $1\leq i(1)<i(2)\dots<i(m)=n$ be given and consider the partition \[\hat0_m=\{\{1,\dots,i(1)\},\dots,\{i(m-1)+1,\dots,i(m)\}\}\in NC(n).\] Consider random variables 
$a_1,\dots,a_n\in \mathcal{A}.$ Then the following equation holds:
\begin{equation}\label{fprod}
\kappa_m(a_1\cdots a_{i(1)},\dots ,a_{i(m-1)+1}\cdots a_{i(m)})=\sum_{\substack{\pi \in NC(n) \\ \pi \vee \hat0_m =1_{n}
}}\kappa_{\mathbf{\pi }}(a_1,\dots,a_n).
\end{equation}
\end{proposition}

It is well known that for a self-adjoint element $a\in\mathcal{A}$ there exist a unique compactly supported measure $\mu_a$ (its distribution) with the same moments as $a$, that is, \[\int_{\mathbb{R}}x^{k}\mu_a (dx)=\tau (a^{k}), \quad \forall k\in \mathbb{N}.\] In particular, if $a\in\mathcal{A}$ is positive (i.e. $a=bb^*$ for some $b\in\mathcal{A}$), the measure $\mu_a$ is supported in the positive half-line. Furthermore, given compactly supported probability measures $\mu_1,\dots,\mu_k$ we can find free elements $a_1,\dots,a_k$ in some $C^*$-probability space such that $\mu_{a_i}=\mu_i$.

For compactly supported measures on the positive half-line $\mu_1,\dots,\mu_k$, we can then consider positive elements $a_1,\dots,a_k\in \mathcal{A}$ such that $\mu_i=\mu_{a_i}$. Positive elements have unique positive square roots, so let $b_1,\dots, b_k$ be the (necessarily free) positive elements such that $b_i^2=a_i$. 

Then, on one hand, the \emph{free multiplicative convolution} is defined as the distribution $\mu_1\boxtimes\cdots\boxtimes \mu_k$ of the positive element $b:=b_1b_2\cdots b_kb_k\cdots b_1$. It is not hard to see that for tracial $C^*$-probability spaces, the moments of $b$ are exactly the moments of $a:=a_1\cdots a_k$. It will be convenient for us to work with $a$ instead of $b$.

On the other hand, the \emph{free additive convolution} is defined as the distribution $\mu_1\boxplus \cdots \boxplus \mu_k$ of the positive element $a_1+a_2+\cdots+a_k$. It is easy to see that $\kappa_n(\mu_1\boxplus \cdots \boxplus \mu_k)=\kappa_n(\mu_1)+\cdots +\kappa_n(\mu_k)$.

 Finally, for a measure $\mu$ and $c>0$ we denote by $D_{c}(\mu)$ the measure such that $D_c(\mu)(B)=\mu(\frac{1}{c}B)$ for all Borel sets $B$. In this case, $\kappa_n(D_c(\mu))=c^n\kappa_n(\mu)$.

\section{Main Formulas}

The following characterization plays a central role in this work. The proof, elementary but cumbersome, can be found in the Appendix.

\begin{proposition} \label{Prop}
i) $\pi\in NC(kn)$ is $k$-preserving if and only if $\pi=Kr(\sigma)$ for some $k$-divisible partition $\sigma\in NC^k(n)$.

ii) $\pi\in NC(kn)$ is $k$-completing if and only if $\pi=Kr(\sigma)$ for some $k$-equal partition $\sigma\in NC_k(n)$.
\end{proposition}

\begin{remark} \label{decom}
In view of the previous characterization, for a $k$-divisible partition $\pi$, the Kreweras complement $Kr(\pi)$ may be divided into $k$ partitions $\pi_1,\pi_2,\dots,\pi_k$, with  $\pi_j$ involving only numbers congruent to $j$ mod $k$. In this case we will write $\pi_1\cup\dots\cup\pi_k=Kr(\pi)$ for such decomposition.
\end{remark}

\begin{figure}
  \begin{center}
    \includegraphics[scale=0.4]{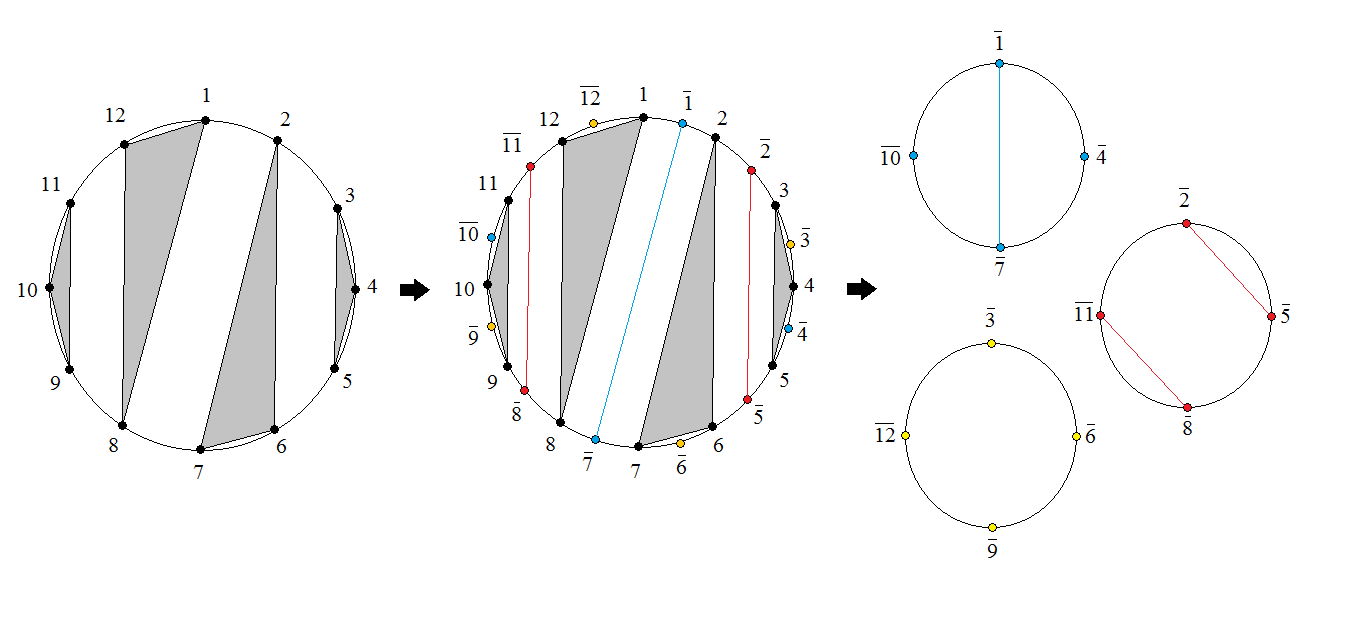}
    \caption{The $3$-equal partition $\pi=\{ \{1,8,12\},\{2,6,7\},\{3,4,5\},\{9,10,11\}\}$  and its Kreweras complement $Kr(\pi)=\pi_1\cup\pi_2\cup\pi_3$, 
with $\pi_1=\{\{1,7\},\{4\},\{10\}\}$, $\pi_2=\{\{2,5\},\{8,11\}\}$ and $\pi_3=\{\{3\},\{6\},\{9\},\{12\}\}$.}
  \end{center}
\end{figure}

We are ready to prove our main Theorem.

\begin{proof}[Proof of Theorem  \ref{main}]
By the formula for products as arguments, we have that \[\kappa_n(a)=\prod_{\substack {\pi \in NC(kn) \\ \pi\vee\rho_n^k=1_{nk}}}\kappa_{\pi}(a_1,\dots ,a_n).\] Since the random variables are free, the sum runs actually over $k$-preserving partitions (otherwise there would be a mixed cumulant). But then by Proposition \ref{Prop} ii), the partitions involved in the sum are exactly the Kreweras complements of $k$-equal partitions, and the formula follows.

For the proof of (\ref{kmom2}), we use the moment-cumulant formula
\[\tau(a^n)=\sum_{\pi\in NC(kn)}\kappa_{\pi}(a_1,\dots ,a_n).\]
Again, the elements involved are free, so only $k$-preserving partitions matter, and these are the Kreweras complements of $k$-divisible partitions by Proposition \ref{Prop} i). Hence the result follows.
\end{proof}

\begin{remark} \label{Boolean}
As pointed out in \cite{BeNi}, Equation (\ref{2prod}) is also satisfied when we replace the free cumulants by the Boolean cumulants. Therefore, Formula (\ref{kcum2}) holds as well for Boolean cumulants $(b_n)_{n\geq 1}$, namely, if $a:=a_1\cdots a_k$ is a product of free random variables, we have
\begin{equation}
b_n(a)=\sum_{\pi \in NC_{k}(n)}b_{Kr(\pi)}(a_1,\dots ,a_k).
\end{equation}

\end{remark}

\begin{corollary} Let $\mu_1,\dots,\mu_k$  be probability measures with positive bounded support and let $\mu=\mu_1 \boxtimes \cdots \boxtimes \mu_k$. We may rewrite our formulas,
\begin{eqnarray}
\kappa_n(\mu)&=&\sum_{\pi \in NC_k(n)}\kappa_{\pi_1}(\mu_1)\dots \kappa_{\pi_k}(\mu_k) \label{kcum},\\
m_n(\mu)&=&\sum_{\pi \in NC^k(n)}\kappa_{\pi_1}(\mu_1)\dots \kappa_{\pi_k}(\mu_k) \label{kmom},
\end{eqnarray}
where $\pi_1\cup\dots\cup\pi_k=Kr(\pi)$ is the decomposition described in  Remark \ref{decom}.
\end{corollary}

\begin{remark} \label{variance}
From Equation (\ref{kmom}) it is easy to see that for compactly supported measures with mean $1$, the variance is additive with respect to free multiplicative convolution, that is \[Var(\mu_1\boxtimes \dots \boxtimes \mu_k)=\kappa_2(\mu_1\boxtimes \dots \boxtimes \mu_k)=\sum_{i=1}^k\kappa_2(\mu_i)=\sum_{i=1}^kVar(\mu_i).\]
\end{remark}

\section{Supports of free multiplicative convolutions}

Our main result can be used to compute bounds for the supports of free multiplicative convolutions of positive measures. Our results in this section generalize those by Kargin in \cite{Kar}, where the case $\mu_1=\dots=\mu_k$ was treated.

We recall that the cardinality of $NC(n)$ is given by the Catalan number $C_n$, which is easily bounded by $4^n$. It is also known that the M\"{o}bius function $Mob:NC(n)\times NC(n)\to \mathbb{C}$ is bounded in absolute value by $C_{n-1}\leq 4^{n-1}$ (see \cite[Prop 13.15]{NiSp}). Thus we are able to control the size of the free cumulants.

\begin{lemma} \label{Lbound}
Let $\mu$ be a probability measure supported on $[0,L]$ with variance $\sigma^2$, such that $E(\mu)=1$. Then $\kappa_2(\mu)=\sigma^2\leq L-1$ and $|\kappa_n(\mu)|<(26L)^{n-1}.$
\end{lemma}

\begin{proof}

It is easy to see that $L\geq 1$ and therefore \[m_n^{\mu}=\int_0^L x^n \mathrm{d}\mu(x)\leq\int_0^L L^{n-1}x \mathrm{d}\mu(x)=L^{n-1}.\] Then we have 
that $\kappa_1(\mu)=1$, $0<\kappa_2(\mu)\leq L-1<26L$, and $|\kappa_3(\mu)|\leq L^2+3L+1<26^2L^2$, and for $n\geq4$ we have
\[|\kappa_n(\mu)|=\sum_{\pi\in NC(n)}|m_{\pi}^{\mu}||Mob[\pi,1_n]|\leq \sum_{\pi\in NC(n)}L^{n-1}4^{n-1}\leq 4^{2n-1}L^{n-1}<(26L)^{n-1},\] since $4^7<26^3.$
\end{proof}

We easily see that the growth of the support is no less than linear. 
\begin{proposition} \label{lowbound}
Let $\mu_1,\dots ,\mu_k$ be compactly supported probability measures on $\mathbb{R}^+$, satisfying $E(\mu_i)=1$, $Var(\mu_i)= \sigma^2$, and
let $L_k$ be the supremum of the support of $\mu:=\mu_1\boxtimes \dots \boxtimes \mu_k$. Then $L_k\geq k\sigma^2+1$.
\end{proposition}
\begin{proof}
It is clear that $E(\mu)=1$, and hence by Remark \ref{variance} we know that $\kappa_2(\mu)=Var(\mu)=k \sigma^2$. By Lemma \ref{Lbound} we have that $\kappa_2(\mu)=k \sigma^2\leq L_k-1$.
\end{proof}

Now, we give an upper bound for the support of the free multiplicative convolution $\mu_1\boxtimes \dots \boxtimes \mu_k$. 

\begin{proposition} \label{support}
There exist a universal constant $C$, such that for all $k$ and all $\mu_1,\dots ,\mu_k$  probability measures supported on $[0,L]$, satisfying $E(\mu_i)=1$, $i=1,\dots ,k$,  the measure $\mu:=\mu_1\boxtimes \dots \boxtimes \mu_k$ satisfies \[Supp(\mu)\subseteq [0,CL(k+1)].\] In general, $C$ may be taken $\leq 26e$. If the measures $\mu_i$, $1\leq i\leq k$ 
have non-negative free cumulants, $C$ may be taken $\leq e$.
\end{proposition}

\begin{proof}
By Equation (\ref{kmom}) we get
\begin{equation}\label{kmom3}
m_n(\mu)=\sum_{\pi \in NC^k(n)}\kappa_{\pi_1}(\mu_1)\dots \kappa_{\pi_k}(\mu_k)
\end{equation}
Since a $k$-divisible partition $\pi\in NC^k(n)$ has at most $n$ blocks, we know that \[|\pi_1|+\dots+|\pi_k|=|Kr(\pi)|=kn+1-|\pi|\geq(k-1)n+1.\] Now, let $\tilde L=26L$. By Lemma \ref{Lbound}, we know that $\kappa_{\pi_i}(\mu_i)\leq (\tilde L)^{n-|\pi_i|}$. Hence
\begin{eqnarray}
\sum_{\pi \in NC^k(n)}\kappa_{\pi_1}(\mu_1)\dots \kappa_{\pi_k}(\mu_k) &\leq& \sum_{\pi \in NC^k(n)} (\tilde L)^{kn-(|\pi_1|+\dots+|\pi_k|)} \\
&\leq &\sum_{\pi \in NC^k(n)}(\tilde L)^{n}\\
&= &\frac{\binom{(k+1)n}{n}}{kn+1}(\tilde L)^{n}
\end{eqnarray}
By taking the $n$-th root and the use of Stirling approximation formula, we obtain that
\begin{eqnarray}
\limsup_{n\to\infty}(m_n(\mu))^{1/n}&=&
\frac{(k+1)^{(k+1)}}{k^k}(\tilde L) \\
&\leq & (k+1)e\tilde L.
\end{eqnarray}
 If $\mu$ has non-negative free cumulants we may replace $\tilde L$ by $L$. 
\end{proof}

One may think from Propositions \ref{lowbound} and \ref{support} that the measures $\nu_k:=D_{1/k}(\mu^k)$ converge to a non-trivial limit. However, since $E[\nu_k]\to0$ and $E[(\nu_k)^2]\to0$, then $\nu_k\to\delta_0$.

\section{More applications and examples}

In this section we want to show some examples of how Theorem \ref{main} may be used to calculate free cumulants.

\begin{example} (Product of free Poissons)
Theorem \ref{main} takes a very easy form in the particular case $\mu_i=\mathrm{m}$, where $\mathrm{m}$ is the Marchenko-Pastur distribution of parameter $1$. Indeed, since $\kappa_n(\mathrm{m})=1$, we get
\begin{equation*}
\kappa_n(\mathrm{m}^{\boxtimes k})=\sum_{\pi\in NC^k(n)}1=\frac{\binom{kn}{n}}{(k-1)n+1},
\end{equation*}
and
\begin{equation*}
m_n(\mathrm{m}^{\boxtimes k})=\sum_{\pi\in NC_k(n)}1=\frac{\binom{(k+1)n}{n}}{kn+1}.
\end{equation*}
Moreover, from the last equation one can easily calculate the supremum of the support $L_k=(k+1)^{k+1}/k^k$.
\end{example}

\begin{example} (Product of shifted semicirculars)
For $\sigma^2\leq \frac{1}{4}$, let $\omega_{+}:=\omega_{1,\sigma^2}$ be the shifted Wigner distribution with mean $1$ and variance $\sigma^2$. The density of $\omega_{+}$ is given by
\begin{equation*}
\omega_{1,\sigma^2}(x)=\frac{1}{2\pi \sigma^2}\sqrt{4\sigma^2-(x-1)^{2}}\cdot 1_{[1-2\sigma,1+2\sigma]}(x)\mathrm{d}x,
\end{equation*}
and its free cumulants are $\kappa_{1}(\omega_{+}) = 1$, $\kappa_{2}(\omega_{+}) =\sigma^2$
and $\kappa_{n}(\omega_{+})=0$ for $n>2.$

 We want to calculate the free cumulants of $\omega_{+}^{\boxtimes k}$, $k\geq 2$. So let $a_1,\dots,a_k$
 be free random variables with distribution $\omega_{+}$. By Theorem \ref{main}, the free
cumulants of $a:=a_1\cdots a_k$ are given by
\begin{eqnarray}
\kappa_{n}(a)&=&\sum_{\pi \in NC^k(n)}\kappa_{Kr(\pi)}(a_1,\dots,a_k).
\end{eqnarray}
If $Kr(\pi)$ contains a block of size greater than $2$, then $\kappa_{Kr(\pi)}=0$. Hence the sum runs actually over $NC(k,n)_{2,1}$. Therefore each summand has the common contribution of $(\sigma^2)^{n-1}$ and by Equation (\ref{sizeH2}) we know the number of summands. Then the free cumulants are \begin{equation} \label{wignercum}
k\frac{((k-1)n)!(\sigma^2)^{n-1}}{(n-1)!((k-2)n+2)!}.
\end{equation}

Note that in this example $m_n(a)\geq \kappa_n(a)$ (since all the free cumulants are positive) and $L=1+2\sigma$. By Proposition \ref{support} and an application of Stirling's formula to Equation (\ref{wignercum}) the supremum $L_k$ of the support 
of $\omega_+^{\boxtimes k}$ satisfies \[(k+1)e(1+2\sigma)\geq L_k =\limsup_{n\to\infty}(m_n(a))^{1/n}\geq\limsup_{n\to\infty}(\kappa_n(a))^{1/n}\geq (k-1)e\sigma^2.\] Hence we obtain a better estimate than the rough bound provided by Proposition \ref{lowbound}.
\end{example}

As another application of our main formula, we show that the free cumulants of $\mu^{\boxtimes k}$ become positive for large $k$. This is of some relevance if we recall that our estimates of the support of $\mu^{\boxtimes k}$ are better with the presence of non-negative free cumulants.

\begin{theorem}
Let $\mu$ be a probability measure supported on $[0,L]$ with mean $\alpha$ and variance $\sigma^2$.  Then for each $n\geq 1$ there exist a constant $N$ such that for all $k\geq N$, the first $n$ free cumulants of $\mu^{\boxtimes k}$ are non-negative.
\end{theorem}

\begin{proof}
Clearly it is enough to show that, for each $n\geq 1$, there exist $N_0$ such that the $n$-th free cumulant of $\mu^{\boxtimes k}$ is positive for all $k\geq N_0$. 

Let $n>1$ and $\tilde\alpha:=\max\{\alpha^{n-1},1\}$. By the same arguments as in Lemma \ref{Lbound} one can show that $|\kappa_n(\mu)|\leq16(L^n)$. Then by Theorem \ref{main} we have
\begin{eqnarray*}
\kappa_n(\mu^{\boxtimes k})&=&
\sum_{\pi \in NC(k,n)_{2,1}}\kappa_{Kr(\pi)}(\mu)+\sum_{\substack{\pi \in NC_{k}(n) \\ \pi \notin NC(k,n)_{2,1}}} \kappa_{Kr(\pi)}(\mu) \\
&\geq & \alpha^{(k-2)n+2}\left(\sum_{\pi \in NC(k,n)_{2,1}}\sigma^{2n-2}-\sum_{\substack{\pi \in NC_{k}(n) \\ \pi \notin NC(k,n)_{2,1}}} (16L)^{n-1}\tilde\alpha\right)\\
&=&\alpha^{(k-2)n+2}(|NC(k,n)_{2,1}|\sigma^{2n-2}-(|NC_{k}(n)|-|NC(k,n)_{2,1}|)(16L)^{n-1}\tilde\alpha).
\end{eqnarray*}
The factor $\alpha^{(k-2)n+2}$ is positive and the rest of the expression becomes positive for all $k$ larger than some $N_0$, since, by Equation (\ref{H2}), $NC(k,n)_{2,1}/NC_k(n)\rightarrow 1$ as $k\rightarrow \infty$.
\end{proof}
It would be interesting to investigate whether or not all free cumulants become positive.

In a recent paper \cite{SaYo}, Sakuma and Yoshida found a probability measure $\mathfrak{h}_{\sigma^2}$, arising from a limit theorem, which is infinitely divisible with respect to both multiplicative and additive free convolutions. Using our methods, we give another proof of this limit theorem. We restrict to the case $E(\mu)=1$. The general case follows directly from this.
\begin{proposition} \label{Sakuma}
Let $\mu$ be a probability measure on supported on $[0,L]$, with $E(\mu)=1$ and $Var(\mu)=\sigma^2$, then \[\lim_{k\to\infty}D_{1/k}\left((\mu^{\boxtimes k})^{\boxplus k}\right)=\mathfrak{h}_{\sigma^2}\]
where $\mathfrak{h}_{\sigma^2}$ is the unique probability measure with free cumulants  $\kappa_n(\mathfrak{h}_{\sigma^2})=\frac{(\sigma^2n)^{n-1}}{n!}$.
\end{proposition}

\begin{proof}
First, we see by Stirling's Formula that 
\begin{equation}
\lim_{k\to\infty}\frac{1}{k^{n-1}}\sum_{\pi \in NC(k,n)_{2,1}}\kappa_{Kr(\pi)}(\mu)=\lim_{k\to\infty}\frac{1}{k^{n-1}}\sum_{\pi \in NC(k,n)_{2,1}}\sigma^{2n-2}=\frac{n^{n-1}}{n!}\sigma^{2n-2}.
\end{equation}
Therefore, by Equation (\ref{H2}), $|NC(k,n)_{2,1}|$ and $|NC_k(n)|$ are of order $k^{n-1}$ as $k\rightarrow\infty$ and \[\frac{|NC_k(n)|-|NC(k,n)_{2,1}|}{|NC_k(n)|}\to 0.\] By using the bound $|\kappa_{n}(\mu)|\leq16(L^n)$, we obtain that \begin{eqnarray}
\lim_{k\to\infty}\frac{1}{k^{n-1}}\sum_{\substack{\pi \in NC_{k}(n) \\ \pi \notin NC(k,n)_{2,1}}} \kappa_{Kr(\pi)}(\mu)\leq \lim_{k\to\infty}\frac{1}{k^{n-1}}\sum_{\substack{\pi \in NC_{k}(n) \\ \pi \notin NC(k,n)_{2,1}}} (16L)^n=0.
\end{eqnarray}
Hence
\begin{eqnarray*}
\lim_{k\to\infty}\kappa_n\left(D_{1/k}\left((\mu^{\boxtimes k})^{\boxplus k}\right)\right)&=&
\lim_{k\to\infty}\frac{1}{k^{n-1}}\kappa_n(\mu^{\boxtimes k})\\&=&
\lim_{k\to\infty}\frac{1}{k^{n-1}}\sum_{\pi \in NC(k,n)_{2,1}}\kappa_{Kr(\pi)}(\mu)
\\&+& \lim_{k\to\infty}\frac{1}{k^{n-1}}\sum_{\substack{\pi \in NC^{k}(n) \\ \pi \notin NC(k,n)_{2,1}}} \kappa_{Kr(\pi)}(\mu)\\
&=&\frac{n^{n-1}}{n!}\sigma^{2n-2}.
\end{eqnarray*}
\end{proof}

\begin{remark}
In view of Remark \ref{Boolean}, some results of this section regarding the free cumulants can be also obtained for the Boolean cumulants. In particular, a second limit theorem from \cite{SaYo}, stating that \[\lim_{k\to\infty}D_{1/k}\left((\mu^{\boxtimes k})^{\uplus k}\right)=\mathfrak{s}_{\sigma^2},\] where $\mathfrak{s}_{\sigma^2}$ is the unique probability measure with Boolean cumulants $b_n(\mathfrak{h}_{\sigma^2})=\frac{(\sigma^2n)^{n-1}}{n!}$, can be proved following the exact same lines as in Proposition \ref{Sakuma}.
\end{remark}

\appendix

\section{Proof of Proposition \ref{Prop}}

\begin{remark} \label{intblock}
A useful characterization of non-crossing partitions is that, for any $\pi\in NC(n)$, one can always find a block $V=\{r+1,\dots,r+s\}$ containing consecutive numbers such that if one removes this block from $\pi$, the partition $\pi\setminus V\in NC(n-s)$ remains non-crossing.
\end{remark}

For a partition $\pi\in NC(n)$ will often write $r\sim_{\pi}s$, meaning that $r,s$ belong to the same block of $\pi$.

Let us introduce two operations on non-crossing partitions. For $n,k\geq 1$ and $r\leq n$, we define $I_r^k:NC(n)\to NC(n+k)$, where $I_r^k(\pi)$ is 
obtained from $\pi$ by duplicating the element in the position $r$, identifying the copies and inserting $k-1$ singletons between the two copies.
More precisely, for $\pi\in NC(n)$, $I_k^r(\pi)\in NC(n+k)$ is the partition given by the relations:

\begin{enumerate}[{\rm(1)}]
\item For $1\leq m_1, m_2 \leq r$,
\[m_1\sim_{I_r^k(\pi)}m_2\ \Leftrightarrow m_1\sim_{\pi}m_2.\]
\item For $r+k \leq m_1, m_2 \leq n+k$,
\[m_1\sim_{I_r^k(\pi)}m_2\ \Leftrightarrow m_1-k\sim_{\pi}m_2-k.\]
\item For $1\leq m_1 \leq r$ and $r+k+1 \leq m_2 \leq n+k$,
\[m_1\sim_{I_r^k(\pi)}m_2\ \Leftrightarrow m_1\sim_{\pi}m_2-k.\]
\item $r\sim_{I_r^k(\pi)}r+k$.
\end{enumerate}

The operation $\tilde I_r^k:NC(n)\to NC(n+k)$ consists of inserting an interval block of size $k$ between the positions $r-1$ and $r$ in $\pi$. We will skip the explicit definition.

The importance of these operations is that they are linked by the relation 
\begin{equation} \label{II}
Kr(I_r^k(\pi))=\tilde I_r^k(Kr(\pi)).
\end{equation}

Our operations preserve properties of partitions, as shown in the following lemma.

\begin{lemma} \label{lema1}
Let $\pi\in NC(nk)$, $r\leq nk$, $s\geq 1$. Then

i) $\pi$ is $k$-preserving if and only if $I_r^{sk}(\pi)$ is $k$-preserving.

ii) $\pi$ is $k$-completing if and only if $I_r^{k}(\pi)$ is $k$-completing.

iii) $\pi$ is $k$-divisible if and only if $\tilde I_r^{sk}(\pi)$ is $k$-divisible.

iv) $\pi$ is $k$-equal if and only if $\tilde I_r^k(\pi)$ is $k$-equal.
\end{lemma}

\begin{proof}
i) By definition of $I_r^k(\pi)$, the relations indicated by $I_r^k(\pi)$ are obtained by relations indicated by $\pi$, with possible shifts by $ks$ (which do not modify congruences modulo $k$). Hence the equivalence follows.

ii) One should think of the block intervals of $\rho_n^k$ as vertices of a graph. For $\pi\in NC(nk)$, an edge will join two vertices $V,W$, if there are elements $r\in V$, $s\in W$ such that $r\sim_{\pi}s$. Then $\pi\vee\rho_n^k=1_{nk}$ if and only if the graph is connected.

It is easy to see that the effect of $I_r^k$ on the graph of $\pi$ is just splitting the vertex corresponding to the block $V$ containing $r$ into $2$ vertices $V_1,V_2$. The edges between all other vertices are preserved, while the edges which were originally joined to $V$ will now be joined either to $V_1$ or $V_2$. Finally, the last additional relation $r\sim_{I_r^k(\pi)}r+k$ means an edge joining $V_1$ to $V_2$. Therefore, it is clear that the connectedness of the two graphs are equivalent.

iii) and iv) are trivial.
\end{proof}

Now we want to show that we can produce all partitions of our interest by applying our operations to elementary partitions.

\begin{lemma} \label{lema2}
i) Let $\pi\in NC(kn)$ be $k$-preserving. Then there exist $m\geq 0$ and numbers $q_0,q_1,\dots ,q_m$, $r_1,\dots ,r_m$ such that
\begin{equation} \label{repn1}
\pi=I_{r_m}^{kq_m}\circ\dots\circ I_{r_1}^{kq_1}(0_{q_0}).
\end{equation}

ii) Let $\pi\in NC(kn)$ be $k$-completing. Then there exist $m\geq 0$ and numbers $r_1,\dots ,r_m$ such that
\begin{equation}\pi=I_{r_m}^{k}\circ\dots\circ I_{r_1}^{k}(0_k).\end{equation}

iii) Let $\pi\in NC(kn)$ be $k$-divisible. Then there exist $m\geq 0$ and numbers $q_0,q_1,\dots ,q_m$, $r_1,\dots ,r_m$ such that
\begin{equation}\pi=\tilde I_{r_m}^{kq_m}\circ\dots\circ \tilde I_{r_1}^{kq_1}(1_{q_0}).\end{equation}

iv) Let $\pi\in NC(kn)$ be $k$-equal. Then there exist $m\geq 0$ and numbers $r_1,\dots ,r_m$ such that
\begin{equation}\pi=\tilde I_{r_m}^{k}\circ\dots\circ \tilde I_{r_1}^{k}(1_k).\end{equation}
\end{lemma}

\begin{proof}
i) We use induction on $n$. For $n=1$ the only $k$-preserving partition is $0_k$, so the statement holds. So assume that i) holds for $n\leq m$. For $\pi \in NC^k(m)$ suppose that there exist $1\leq r<r+sk \leq km$ such that $r\sim_{\pi}r+sk$ and $r+1,\dots r+sk-1$ are singletons of $\pi$ (if no such pair $(r,s)$ exist, necessarily $\pi=0_{mk}$ and we are done). Then its easy to see that $\pi=I_r^{sk}(\pi')$ for some $\pi'\in NC((n-s)k)$. By Lemma \ref{lema1} i) $\pi'$ is $k$-preserving. By induction hypothesis $\pi'$ has a representation as in Equation (\ref{repn1}) and hence, so does $\pi=I_r^{sk}(\pi')$.

The proof of ii) is similar. The proofs of iii) and iv) are trivial using Remark \ref{intblock}.
\end{proof}

\begin{proof}[Proof of Proposition \ref{Prop}]
We only show the first implication of i). The converse and ii) are similar.

Let $\pi\in NC(kn)$ be $k$-preserving. Then by Lemma \ref{lema2} i) we can express it as \[\pi=I_{r_m}^{kq_m}\circ\dots\circ I_{r_1}^{kq_1}(1_{q_0}).\]
But then we can apply Equation (\ref{II}) at every step, obtaining
\begin{eqnarray}
Kr(\pi)&=&Kr(I_{r_m}^{kq_m}\circ\dots\circ I_{r_1}^{kq_1}(0_{q_0})) \\
&=&\tilde I_{r_m}^{kq_m}\circ Kr(I_{r_{m-1}}^{kq_{m-1}}\dots \circ I_{r_2}^{kq_2}\circ I_{r_1}^{kq_1}(1_{q_0})) \\
&=&\tilde I_{r_m}^{kq_m}\circ \tilde I_{r_{m-1}}^{kq_{m-1}}\circ Kr(I_{r_{m-2}}^{kq_{m-2}}\circ \dots \circ I_{r_2}^{kq_2}\circ I_{r_1}^{kq_1}(1_{q_0})) \\
&& \vdots \\
&=&\tilde I_{r_m}^{kq_m}\circ \dots \circ \tilde I_{r_2}^{kq_2}\circ \tilde I_{r_1}^{kq_1}(Kr(1_{q_0})) \\ 
&=&\tilde I_{r_m}^{kq_m}\circ \dots \circ \tilde I_{r_2}^{kq_2}\circ \tilde I_{r_1}^{kq_1}(0_{q_0}),
\end{eqnarray}
which, by Lemma \ref{lema1} iii) is $k$-divisible.
\end{proof}


\end{document}